\documentclass[11pt,oneside]{amsart}
\usepackage{amssymb, amscd, amsmath, amsthm, latexsym, hyperref,xcolor}
\usepackage{pb-diagram, fancyhdr, graphicx, psfrag}
\usepackage[text={6.3in,8.5in},centering,letterpaper,dvips]{geometry}
\usepackage[
textwidth=0.9in,
backgroundcolor=yellow,
bordercolor=orange,
textsize=scriptsize]{todonotes}

\def\cs{\mathop{\#}}

\def\calf{\mathcal{F}}

\def\calk{\mathcal{K}}

\def\calp{\mathcal{P}}

\def\call{\mathcal{L}}

\def\calb{\mathcal {B}}
\def\calz{\mathcal {Z}}
\def\cfk{{\textrm{CFK}}}

\newcommand{\spinc}{\ifmmode{{\mathfrak s}}\else{${\mathfrak s}$\ }\fi}
\newcommand{\spinct}{\ifmmode{{\mathfrak t}}\else{${\mathfrak t}$\ }\fi}
\newcommand{\spincw}{\ifmmode{{\mathfrak w}}\else{${\mathfrak w}$\ }\fi}

\def\F{\mathbb F}

\def\La{\Lambda}
\def\alg{{\rm alg}}
\def\alex{{\rm Alex}}

\newcommand{\fig}[2] { \includegraphics[scale=#1]{#2} }

\def\rmc{\mathrm{C}}
\def\rma{\mathrm{A}}
\def\cfki{\cfk^\infty(K)}

\newtheorem{theorem}{Theorem}[section]

\newtheorem*{theorem*}{Theorem}

\newtheorem{corollary}[theorem]{Corollary}

\newtheorem{proposition}[theorem]{Proposition}
\theoremstyle{definition}
\newtheorem{definition}[theorem]{Definition}

\newtheorem{example}[theorem]{Example}
\theoremstyle{remark}

\numberwithin{equation}{section}

\def\U{\Upsilon}
\def\Utwoprime{\U_{K,t}^{2}\hskip -1.5ex\raisebox{.5ex}{$^\prime$}\hskip 0.9ex}

\begin{document}


\title[Secondary Upsilon Invariants]{Secondary Upsilon Invariants of Knots}
\author{Se-Goo Kim}
\author{Charles Livingston}
\address{Se-Goo Kim: Department of Mathematics and Research Institute for Basic Sciences, Kyung Hee University, Seoul 02447, Korea }
\email{sgkim@khu.ac.kr}
\address{Charles Livingston: Department of Mathematics, Indiana University, Bloomington, IN 47405 }
\email{livingst@indiana.edu}

\thanks{The first author was supported by the Basic Science Research Program through the National Research Foundation of Korea (NRF) funded by the Ministry of Education (NRF-2015R1D1A1A01058384). The second author was supported by a Simons Foundation grant and by NSF-DMS-1505586.}

\maketitle


\begin{abstract}  The knot invariant Upsilon, defined by Ozsv\'ath, Stipsicz, and Szab\'o,  induces a homomorphism from the smooth knot concordance group to the group of piecewise linear functions on the interval [0,2].  Here we define a set of related secondary invariants, each of which assigns to a knot a piecewise linear function on [0,2].  These secondary invariants provide bounds on the genus and concordance genus of knots.  Examples of knots for which Upsilon vanishes but which are detected by these secondary invariants are presented.
\end{abstract}

\section{Introduction}
In~\cite{oss}, Ozsv\'ath, Stipsicz, and Szab\'o defined the {\it Upsilon} invariant of knots $K \subset S^3$, denoted $\U_K(t) $.   The map $K \to \U_K$ is a homomorphism from the smooth knot concordance group to the group of piecewise linear functions on the interval $[0,2]$.  This homomorphism provides bounds on the three-genus, the four-genus, and the concordance genus of knots.   Since then, this invariant has been used to effectively address a range of problems; for a sampling, see~\cite{borodzik-hedden, borodzik-livingston, feller-krcatovich, feller-park-ray, kim-wu,   livingston-cott, oss2, wang}.

For each $t \in (0,2)$, we will  define a secondary function denoted  $\U_{K,t}^2(s)$.  This is again a piecewise linear function on the interval $[0,2]$ and is a concordance invariant that provides bounds on the three-genus and concordance genus of knots.  For any piecewise linear function $f(t)$ on $[0,2]$, the jump function of the derivative, $\Delta f' (t)$, is well-defined for all $t \in (0,2)$.  If  $\Delta \U'_K(t) > 0 $, then $\U^2_{K,t}(s) < 0 $ for all $s$.
The most interesting case is when  $\Delta \U'_K(t)= 0 $ for all $t$ (that is, when $\U_K(t)$ is identically 0)  but  $\U^2_{K,t}(s) $ is a nontrivial function for some values of $t$.  

Jen Hom~\cite{hom-epsilon-upsilon} constructed a knot $K$ with vanishing $\U$--invariant  that can be shown not to be slice using the  $\epsilon$--invariant defined in~\cite{hom1}.  We will show that $K$ can also be quickly  shown to be nontrivial using $\U^2$. As an added feature,  a  corollary of the genus bounds imposed by $\U^2$ is that $g_c(nK) \ge 4n -2$.
In particular, this yields  example for which  $g_c$ can be arbitrarily large but have  vanishing $\U$.

As a final set of examples,  we  describe an infinite set of complexes for which both $\U$ and $\epsilon$ vanish, but which  can be shown to be independent using $\U^2$; whether these complexes arise from actual knots has not been determined.

\section{Knot complexes, $\cfki $}
To each knot $K \subset S^3$, there is an associated bifiltered graded chain complex $\cfki$ which has a compatible structure as a $\F[U,U^{-1}]$--module, where $\F$ is the field with two elements.  We abbreviate $\F[U,U^{-1}]$ by $\La$.  As we   explain in this section, Figure~\ref{fig34-labeled} provides a schematic illustration of the complex $\cfki$ associated to the torus knot $K = T_{3,4}$.  

\begin{figure}[h]
\fig{.25}{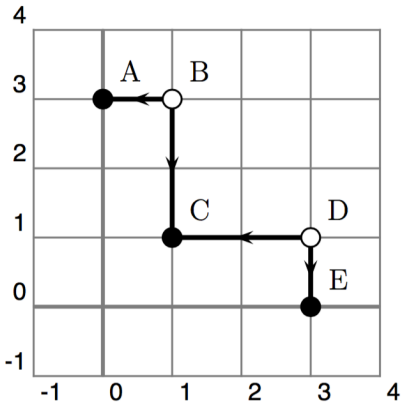}  
\caption{$\cfk^\infty(T(3,4))$ } 
\label{fig34-labeled}
\end{figure}

We now write $\rmc$ for $\cfki$. In our example, there are three black dots in the diagram; these represent  basis elements  for the set of elements of grading 0, $\rmc_0$.  The two white dots represent a  basis for $\rmc_1$. The arrows represent the boundary map: for instance, $\partial C = 0$ and $\partial B = A + C$.
The complex is bifiltered.  The coordinates $(i,j)$ of a vertex represent the {\it algebraic} and {\it Alexander} filtration levels of each of the generators.  For any element in the complex, we write   $\alg(x)$ and $\alex(x)$ for the two filtration levels.

The complex as drawn represents, in the notation of~\cite{hendricks-manolescu}, a model complex for $\rmc$ in which the $\La$--structure is hidden.  The full complex is formed by taking all integer diagonal translates of the illustrated complex; the action of $U$ shifts the vertices a distance of one down and to the left.  More formally, the diagram illustrates a finite dimensional complex $\rmc'$ and $\rmc = \rmc' \otimes \La$, graded and filtered so that the action of $U$ lowers gradings by 2 and filtration levels by 1.

All the knot complexes we consider can be described by schematic diagrams of such model complexes.  In each case, the model $\rmc'$ has the property that its homology is $H(\rmc') \cong \F$, with the generator in grading 0.  Furthermore, the minimal algebraic filtration level of representatives of the nontrivial homology class is 0; similarly for the Alexander filtration level.  Thus, if the  diagram is connected,   it determines the grading level of all vertices.

Typically, there can be more than one vertex at a given bifiltration level.  Since we cannot illustrate more than one vertex at a given lattice point $(i,j)$, we will place such vertices in the unit square with bottom left corner at $(i,j)$.      Figure~\ref{figure8} illustrates a model complex $\rmc$ for the Figure 8 knot, $4_1$.  The two vertices at bifiltration level $(0,0)$ are both cycles representing a generator of $H_0(\rmc)$.

\begin{figure}[h]
\fig{.25}{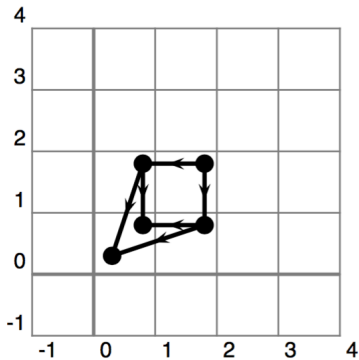}  
\caption{The figure eight knot: $\cfk^\infty(4_1)$ } 
\label{figure8}
\end{figure}

\subsection{Formal knot complexes}  The essential features of $\rmc = \cfki$ are as follows.  It is a doubly filtered graded complex with a compatible $\Lambda$--module structure.  As a $\Lambda$--module, $\rmc$ is finitely generated and free.  The action of $U \in \Lambda$   lowers filtration levels by one and gradings by two.  The homology is given by  $H(\rmc) \cong \Lambda$ with $1 \in \Lambda$ representing a generator of $H_0(\rmc)$.  Finally, there is a symmetry property:  the complex formed by switching the two filtrations is chain homotopy equivalent to $\rmc$ by a filtered graded chain homotopy equivalence. 

All of our work applies to any complex with these properties.  We will refer to them as $\calk$--complexes.  

\subsection{Concordance of complexes} The following result is proved by Hom in~\cite{hom-survey}.

\begin{proposition}\label{homstablethm} If $K$ and $J$ are concordant knots, then there are acyclic complexes $\rma_1$ and $\rma_2$ such that  $\cfk^\infty(K) \oplus \rma_1 = \cfk^\infty(J) \oplus \rma_2$.
\end{proposition}

Here the $\rma_i$ are acyclic as graded complexes but not necessarily as filtered complexes.   We define two $\calk$--complexes to be \emph{concordant} if they are similarly stably equivalent.  A standard argument shows that the set of concordance classes forms an abelian group under tensor products.   

\section{Upsilon, $\U_K$} 

In defining $\U_K$, we follow the presentation  in~\cite{livingston1}.  For any $t \in [0,2]$ and $s\in \mathbb{R}$, we define the subcomplex $\calf_{t,s} \subset \cfki$ as follows.  Let $\mathcal{B}$ denote a bifiltered graded basis of $\cfki$.  
$$\calf_{t,s} =   \left< \{x \in \mathcal{B}\  \big|\   (t/2)Alex(x) + (1-t/2)Alg(x) \le s \}\right>.$$ 
This is independent of the choice of bifiltered graded basis.

In Figure~\ref{fig34shaded} we illustrate the complex $\cfk^\infty(T_{3,4})$ with a half-space shaded in.  This shaded region represents the subcomplex $\calf_{2/3, 1}$.  
In general, $\calf_{t,s}$ is represented by a half-space with boundary a line of slope $m = 1- 2/t$;  a half-space with boundary of slope $m$ corresponds to $\calf_{t,s}$ with $t = 2/({1-m})$.  Given the bounding line, the value of $s$ is given by $tj_0/2$, where $j_0$ is the $j$--intercept of the line.  We call the bounding line for the half-space the 
{\it support line} for $\calf_{t,s}$, which we denote $\call_{t,s}$.

\begin{figure}[h]\label{fig:doublefilter}
\fig{.2}{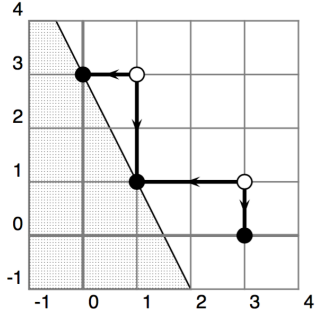}
\caption{$\cfk^\infty(T(3,4))$ doubly filtered}
\label{fig34shaded}
\end{figure}

Given this, we can now define $\U_K(t)$ for any $t \in [0,2]$.

\begin{definition} Let 
$$\gamma_K(t) =  \min\{ s \ | \ H_0(\calf_{t,s} ) \to H_0(\cfki) \cong \F \text{ is surjective}\}.$$   Define $$\U_K(t) = -2 \gamma_K(t).$$
\end{definition}

That $\U_K$ is well defined follows readily from~Proposition~\ref{homstablethm}.

\begin{corollary} If $K$ and $J$ are concordant, then $\U_K = \U_J$.
\end{corollary}

\subsection{Vertices and pivot points} \label{pivots} Let $\calp$ denote the set of bifiltration levels of elements of $\cfki$.  (This depends on the homotopy representative of the complex.) Fix a value of  $t \in [0,2]$ and recall that $\gamma_K(t) = -\U_K(t)/2$.  The line $\call_{t, \gamma(t)}$ contains a nonempty subset of $\calp$ which we denote $\calp_t$.  In the illustrated example, the subset is $\calp_{2/3} = \{(0,3), (1,1)\}$. For small values of $\delta$, $\calp_{t-\delta}$ contains exactly one element of $\calp_t$ and $\calp_{t + \delta}$ also contains exactly one element of $\calp_t$.  We call these the \emph{negative} and \emph{positive pivot points} at $t$, $p_t^-$ and $p_t^+$.  The next result is proved in~\cite{livingston1}, although the statement there places the difference in absolute value.

\begin{theorem}\label{thm:upsilon_at_signularity}  The function $\U_K(t)$ has a singularity at $t$ if and only if $p_t^- \ne p_t^+$.  In general, for $t \in (0,2)$, $\Delta \U_k'(t) =\frac{2}{t}\left( i(p_t^+) - i(p_t^-)\right)$  where $i(p)$ denotes the first coordinate of a lattice point~$p$.
\end{theorem}

\begin{example}\label{example34+} As an example, Figure~\ref{cfkt34} illustrates the complexes for the torus knots $T(3,4)$ and $-T(3,4)$.  Black dots are at grading level 0, white dots at grading 1, and gray dots at grading $-1$.  (The entire complexes include all the $U^k$ translates of these diagrams.)

\begin{figure}[h]
\fig{.25}{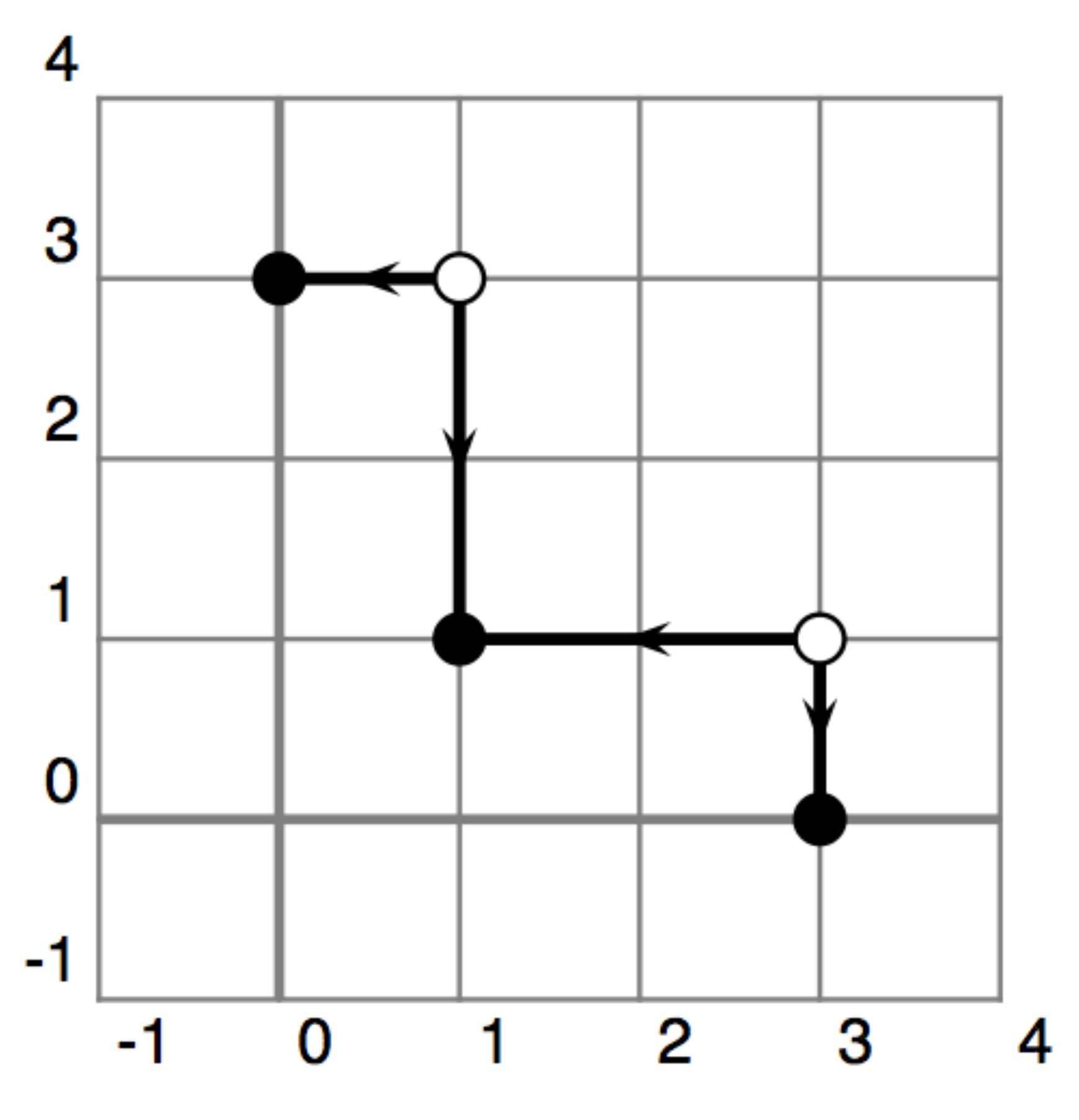} \hskip1in \fig{.25}{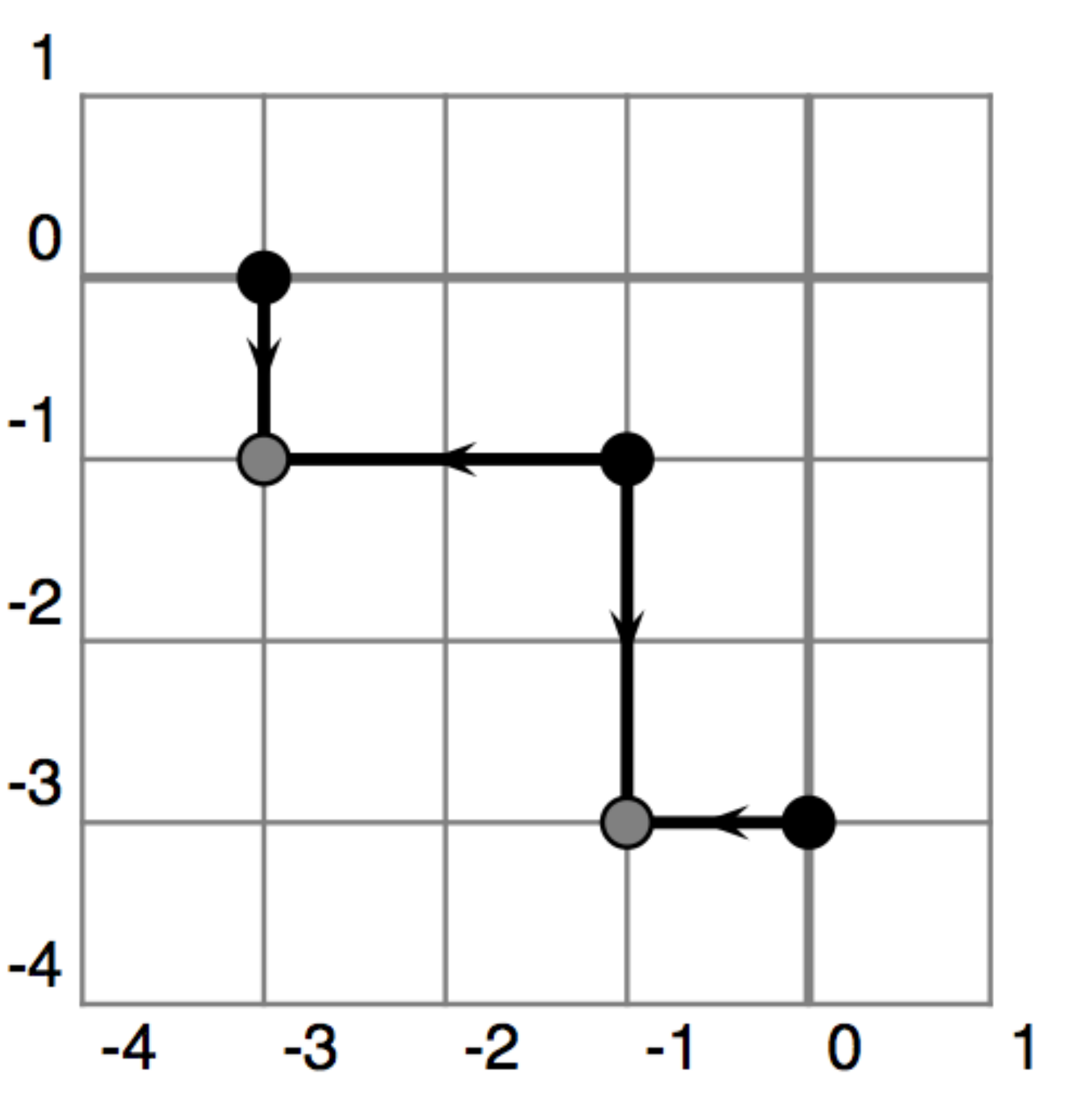}
\caption{$\cfk^\infty(T(3,4))$ and $\cfk^\infty(-T(3,4))$}
\label{cfkt34}
\end{figure}

For $T(3,4)$, each of the black dots represents a generator of $H_0(\cfki)$.  For lines of slopes between $-\infty$ and $-2$, that is, for $t \in [0,2/3]$,  those that go through the vertex $(0,3)$ contain one of the generators, and no line of lesser $y$--intercept does.    Thus, for $t$ in this interval, we have $$\U_{T(3,4)}(t) = -2(3t /2) = -3t.$$
For $t\in[2/3,1]$, the line of slope $m = 1-2/t$ that contains one of the generators goes through the vertex at $(1,1)$.  This line has $y$--intercept $y_0 = 1 -m = 2/t$, and thus  $$\U_{T(3,4)}(t) = -2(y_0t /2) = -2(2 /2)  = -2.$$

\end{example}

\begin{example} For the knot $-T(3,4)$, the only cycle representing a generator of $H_0(\cfk^\infty(K))$ is the sum of the generators at bifiltration levels $(-3,0), (-1,-1)$ and $(0,-3)$.   If $t \in [0,2/3]$ (that is, if the corresponding line has slope less than or equal to $-2$), then for $\calf_{t,s}$ to contain   the generator, it must contain the vertex $(0,-3)$ and so has $y$--intercept $-3$.   
Thus, 
$$\U_{-T(3,4)}(t) = -2(y_0t /2) = -2(-3t /2)  =  3t.$$  If $t\in [2/3, 1]$, the corresponding lines have slope $m \ge -2$, and the line of this slope with least $y$--intercept which bounds a lower half-space that contains a  generator must contain the vertex $(-1,-1)$.  This line has $y$--intercept $y_0 = -1 + m = -1+ 1- 2/t = -2/t$.  For these values of $t$, 
$$\U_{-T(3,4)}(t) = -2(y_0t /2) = -2(-2 /2)  =  2.$$
This is consistent with the fact that $\U_{-K}(t) = - \U_K(t)$.

\end{example}

\section{Defining $\U^2_{K,t}$}\label{definegamma2}

We continue to work with a fixed basis $\calb$ so that chains are represented by a collection of vertices, but note that it is easily checked that the definitions are independent of the choice.  

For a fixed value of $t \in (0,2)$, recall we have by definition  $\gamma_K(t) = -\U_K(t)/2$. Choose a $\delta$ small enough to define $p^-_t$ and $p^+_t$, as in Section~\ref{pivots}, and  let $t^{\pm} = t \pm \delta$.  Let $\calz^\pm $ be the set of cycles in $\calf_{t^\pm,\gamma_K(t^\pm)}$ which represent nontrivial elements in $H_0(\cfki)$.  Notice that each element of $\calz^\pm$ is represented by a set of  vertices that includes one at the lattice point $p^\pm_t$. 

\begin{theorem}\label{thm:Zpm_disjoint}
If $\Delta \U'_K(t) > 0$, then the sets $\calz^-$ and $\calz^+$ are disjoint.  
\end{theorem}

\begin{proof}
Let $z^-$ be a cycle in $\calz^-$; as noted above, $z^-$ is represented by a set of basis elements that includes one at $p^-_t$.
The hypothesis $\Delta \U'_K(t) > 0$ combined with Theorem~\ref{thm:upsilon_at_signularity} implies  that $p_t^-\ne p_t^+$ and $p_t^-$ is not contained in $\calf_{t^+,\gamma_K(t^+)}$. So, $z^-$ is not in $\calz^+$. Similary, one can show the converse:  no element of $\calz^+$ belongs to $\calz^-$. This  completes the proof.
\end{proof}

\begin{definition} \label{def:gamma2}

Let $\calz^\pm = \{z^\pm_j\}$.  For each $s \in [0,2]$, we set $\gamma^2_{K,t}(s)$ to be the minimum value of $r$ such that some $z^-_i$ and $z^+_k$ represent the same homology class in $H_0(\calf_{t, \gamma_K(t)} + \calf_{s,r})$.
\end{definition}
Note that if $\calz^-$ and $\calz^+$ are not disjoint, then $\gamma^2_{K,t}(s)=-\infty$.

This definition is made more clear with an example. 

\begin{example}  Figure~\ref{fig34double} illustrates the complex for $T(3,4)$.  The half-space with boundary line containing the vertices $(0,3) $ and $(1,1)$ is shaded in.  The slope of its support line is $-2$, and thus it corresponds to the value of $\U_K$ at $t = \frac{2}{1- (-2)} = \frac{2}{3}$.  Recall $\U_K(\frac{2}{3} )= -2$. In this case, $\calz^\pm$ each contains exactly one element, represented by vertices at $p^-$ and $p^+$, with coordinates $(0,3)$ and $(1,1)$.

The second line in the figure, through $(1,3)$,  was chosen to have slope $-\frac{1}{4}$.  The corresponding value of $s$ is $s = \frac{2}{1-(-1/4)} = \frac{8}{5}$.  The particular line drawn is the line of slope $-\frac{1}{4} $ with the least $j$--intercept for which the classes represented by $p^\pm$ are equal in $H_0(\calf_{t, u(t)} + \calf_{s,r})$.  (This subcomplex is represented by the union of the two shaded half-spaces.)  To compute the value of $r$, rather than find the $j$--intercept, we compute the value of $\frac{s}{2}(j) + (1 -\frac{s}{2})(i) $  at $(1,3)$.  The value of $r$ is given by
$$r = ((8/5)/2)(3) + ((1 - (8/5)/2)(1) = 13/5.$$
We now have that $\gamma^2_{K,\frac{2}{3}}\left(\frac{8}{5}\right) = \frac{13}{5}$.  

\begin{figure}[h]\label{fig:doublefilter}
\fig{.2}{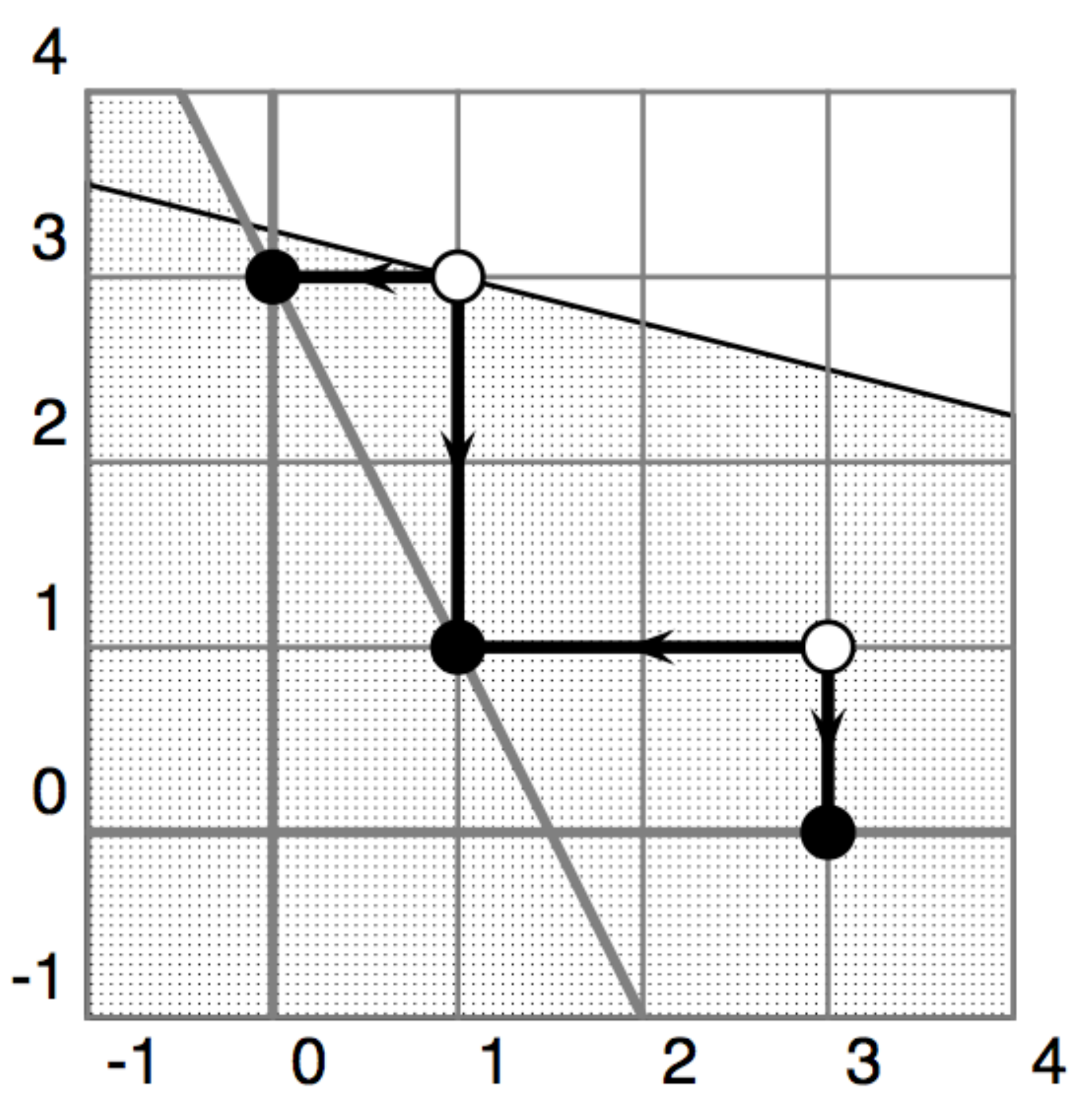}
\caption{$\cfk^\infty(T(3,4))$ doubly filtered}
\label{fig34double}
\end{figure}
\end{example}

\begin{definition}  $$\U^2_{K,t}(s)= -2\gamma^2_{K,t}(s) - \U_{K,t} = -2(\gamma^2_{K,t}(s) - \gamma_K(t)).$$
\end{definition}

\begin{example} For $K= T(3,4)$, $\U_K(t)$ has  a singularity at $t = \frac{2}{3}$; at that singularity the slope has  a positive jump.   Note that $\U_K(\frac{2}{3}) = -2$.  We compute that for $s\in [0,2]$ the line through $(1,3)$ gives
$$r = (s/2)(3) + (1 -s/2)(1) =  1 + s = \gamma^2_{K,\frac{2}{3},s}.$$  It follows that 
$$\U^2_{K,\frac{2}{3}}(s) = -2(1+s) - (-2) = -2s.$$ \end{example}

\begin{example} A more interesting example is provided by the torus knot $K = T(5,7)$, with $\cfk^\infty(K)$ as illustrated in Figure~\ref{fig:t57}.

\begin{figure}[h]
\fig{.35}{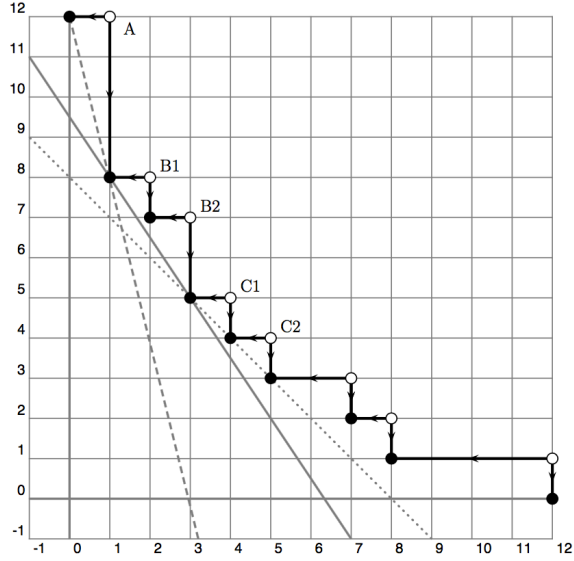}
\caption{$\cfk^\infty(T(5,7))$ }
\label{fig:t57}
\end{figure}
In the figure, support lines of slope $m = -4, -3/2, $ and $-1$ are drawn, and these correspond to the singular values of $\U_K$ at $t= 2/5, 4/5$ and $1$.
A computation quickly shows that 
$$
\U_K(t)  =   \begin{cases}
-12t &\mbox{if } 0 \le t \le 2/5\\
-2 -7t    &\mbox{if } 2/5 \le t \le 4/5\\
-6 -2t    &\mbox{if } 4/5 \le t \le 1.
\end{cases} 
$$

We now consider $\U^2_{K,t}$: for $t= 2/5$ (where the relevant point is labeled $A$ in figure, at $(1,12)$; for $t=4/5$, where the relevant points are labeled $B1$ and $B2$ at $(2,8)$ and $(3,7)$; and for $t=1$, where the relevant points are labeled $C1$ and $C2$ at $(4,5)$ and $(5,4)$.   It is  a straightforward computation to show that:

$$
\U^2_{K,2/5}(s)  =   
\tfrac{14}{5} -11s \mbox{ if } 0 \le s \le 2.
$$

$$
\U^2_{K,4/5}(s)  =   \begin{cases}
\frac{8}{5} -4s &\mbox{if } 0 \le s \le 1\\
\frac{18}{5} -6s    &\mbox{if } 1 \le s \le 2.
\end{cases} 
$$

$$
\U^2_{K,1}(s)  =   \begin{cases}
-2+s &\mbox{if } 0 \le s \le 1\\
-s    &\mbox{if } 1 \le s \le 2.
\end{cases} 
$$

\end{example}

\begin{example}
If $\Delta \U'_K(t) > 0$, then  $\U^2_{K,t}(s)<0$ for some $s$ by Theorem~\ref{thm:Zpm_disjoint}. In fact, $\U^2_{K,t}(t)<0$ in this case. If $\Delta \U'_K(t) \le 0$, there are no known conditions on the signs of $\U^2_{K,t}(s)$. 
There are knots with $\Delta \U'_K(t) < 0$ and $\U^2_{K,t}(s)=\infty$.  For example, $K$ can be the mirror image of $T(3,4)$ or $T(5,7)$. Note that these knots have intersecting $\calz^\pm$. 

On the other hand, there are $\calk$--complexes $\rmc$ with $\Delta \U'_{\rmc}(t) < 0$ and $\U^2_{\rmc,t}(s)<0$. For example, consider the complex in Figure~\ref{fig:U'dec}.  For $t=1$, $z^-$ is the sum of the generators at $(-3,1)$ and $(0,-2)$; $z^+$ is the sum of the generators at $(-2,0)$ and $(1,-3)$.  This complex  has $\Delta \U'_{\rmc}(1)=-4$ and $\U^2_{\rmc,1}(s)=-4$.

\begin{figure}[h]
\fig{.3}{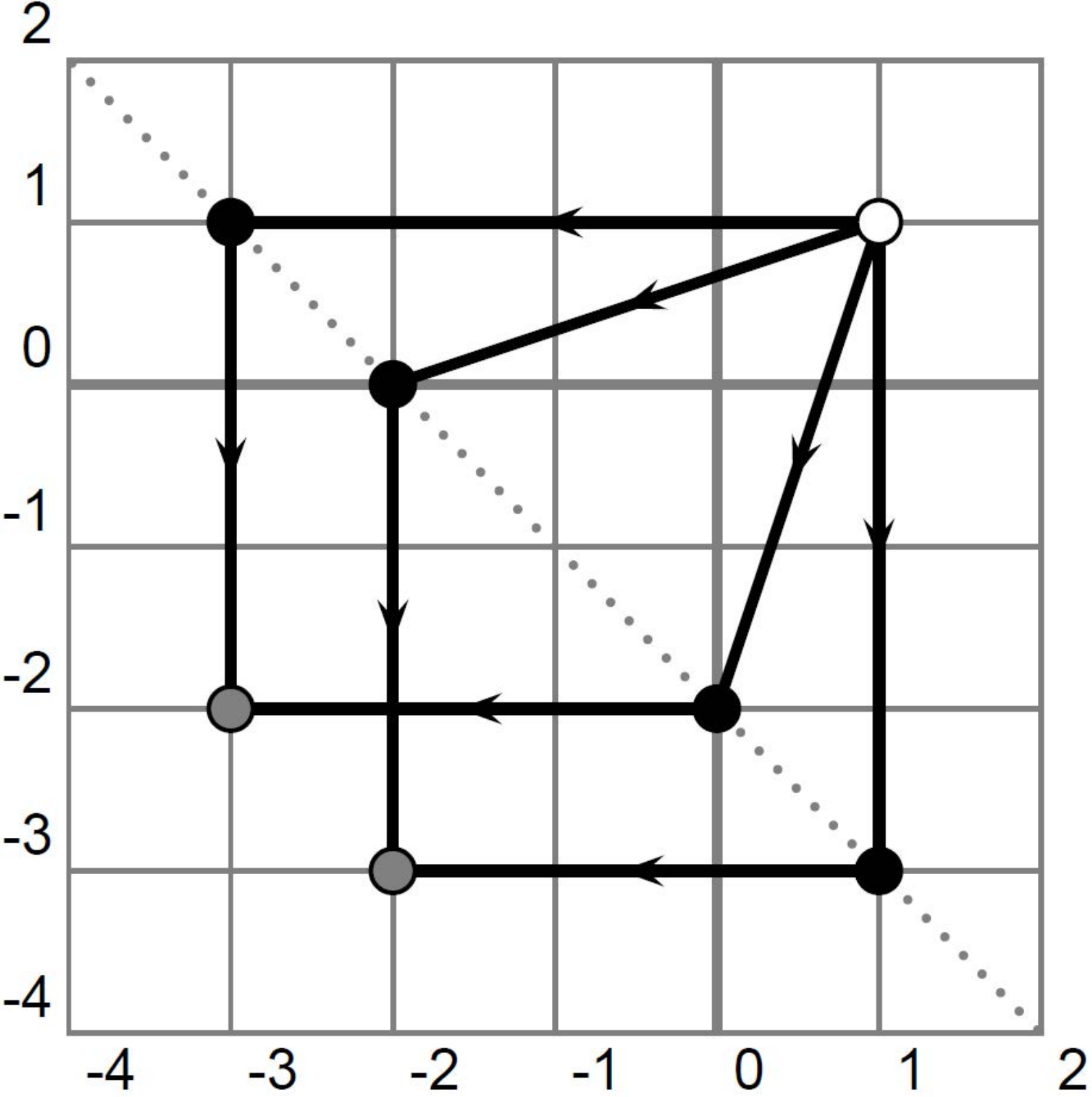}
\caption{A $\calk$--complex $\rmc$ with $\Delta \U'_{\rmc}(1)<0$ and $\U^2_{\rmc,1}(s)<0$ }
\label{fig:U'dec}
\end{figure}
\end{example}

\subsection{Concordance invariance}
The definitions for $\U$ and $\U^2$ can be extended to $\calk$--complexes.
\begin{theorem} Adding acyclic summands to a $\calk$--complex does not change the value of  $\U^2$. Therefore, $\U^2$ is a concordance invariant of knots.
\end{theorem}

\begin{proof}
Let $\rmc$ be a $\calk$--complex and $\rma$ an acyclic complex. It is obvious that $\gamma^2_{\rmc\oplus \rma,t}(s)\le \gamma^2_{\rmc,t}(s)$ from the definition. To show the reversed inequality, let $z^-$ and $z^+$ be two cycles representing the same homology class in $H_0$ as in Definition~\ref{def:gamma2}. Let $c$ be a chain such that $\partial c=z^--z^+$. We can write $z^\pm=z^\pm_{\rmc} + z^\pm_{\rma}$ and $c=c_{\rmc} + c_{\rma}$, where $z^\pm_{\rmc}$ and $c_{\rmc}$ are elements of $\rmc$, and $z^\pm_{\rma}$ and $c_{\rma}$ are elements of $\rma$. Then $\partial c_{\rmc}=z^-_{\rmc}-z^+_{\rmc}$ and $\partial c_{\rma}=z^-_{\rma}-z^+_{\rma}$. Since each $z^\pm_{\rmc}$ belongs to the corresponding subcomplex $\calf_{t^\pm,\gamma_{\rmc}(t)}$ of $\rmc$, this implies that $\gamma^2_{\rmc,t}(s)\le \gamma^2_{\rmc\oplus \rma,t}(s)$, as desired. Finally, by Proposition~\ref{homstablethm} we see that $\U^2$ is a concordance invariant of knots.
\end{proof}

\begin{example} In~\cite{hom-epsilon-upsilon}, Hom constructed the complex $\cfk^\infty(T(4,5) \cs -T(2,3;2,5))$.   In Figure~\ref{fig:ladders} we illustrate this complex (modulo a missing acyclic summand), along with  the complexes for $T(2,5)$ and $-T(2,5)$.  We abbreviate these by $\rmc_1$, $\rmc_2$ and $-\rmc_2$.  For both  $\rmc_1$ and $\rmc_2$ we have
\begin{center}
$\U_K(t) = \begin{cases}
-2t &\mbox{if } t \le 1\\
-4+2t &\mbox{if } t\ge 1.
\end{cases}
$
\end{center}
In both cases, the only singularity is at $t =1$, at which the pivot points are $p^- = (0,2)$ and $p^+ = (2,0)$;   there is only one cycle in each of $\calz^\pm$, and these are  represented by a single basis element, located at either $ (0,2)$ or $  (2,0)$.  

\begin{figure}[h]
\fig{.15}{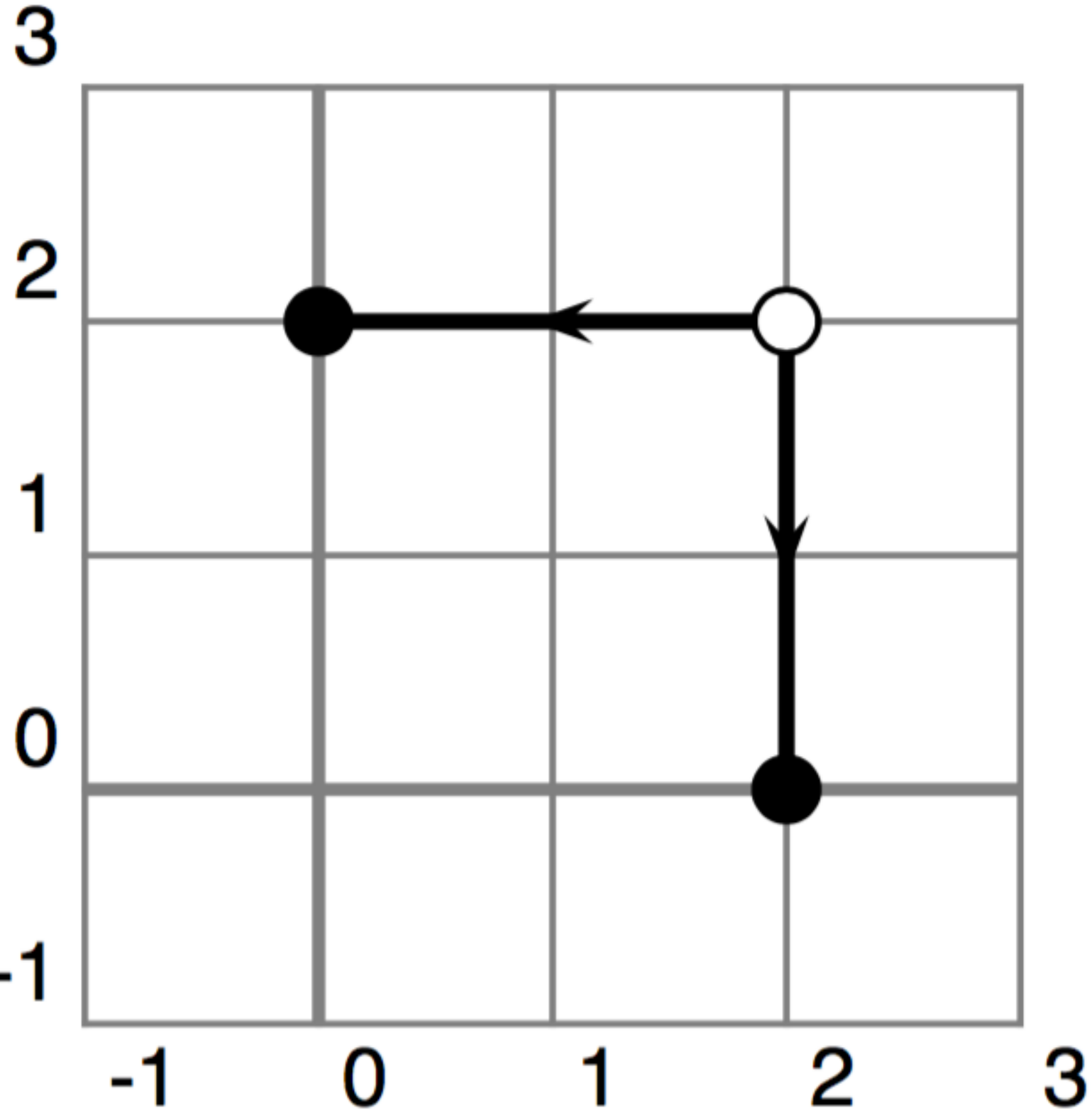} \hskip.5in  \fig{.15}{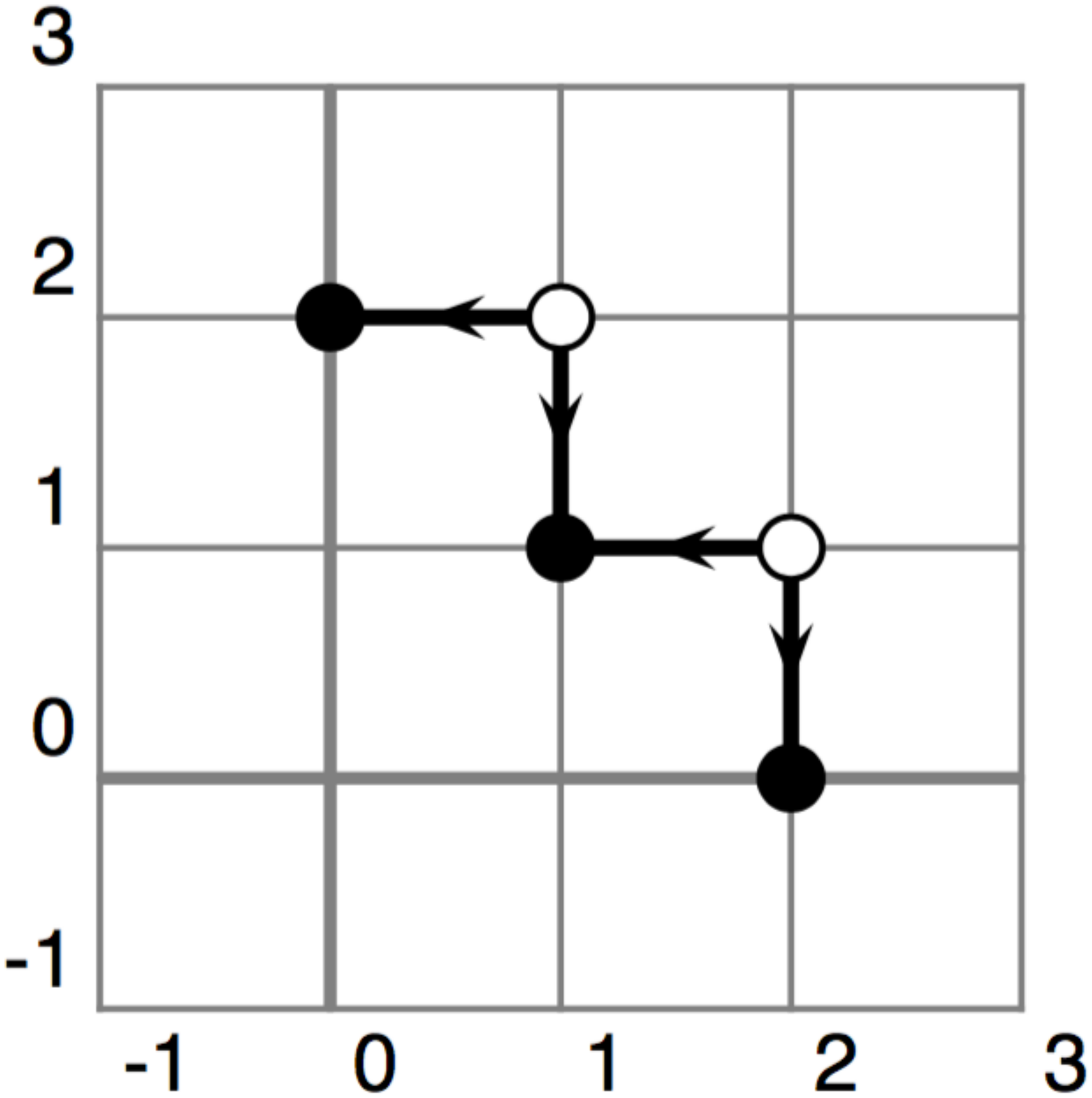} \hskip.5in \fig{.15}{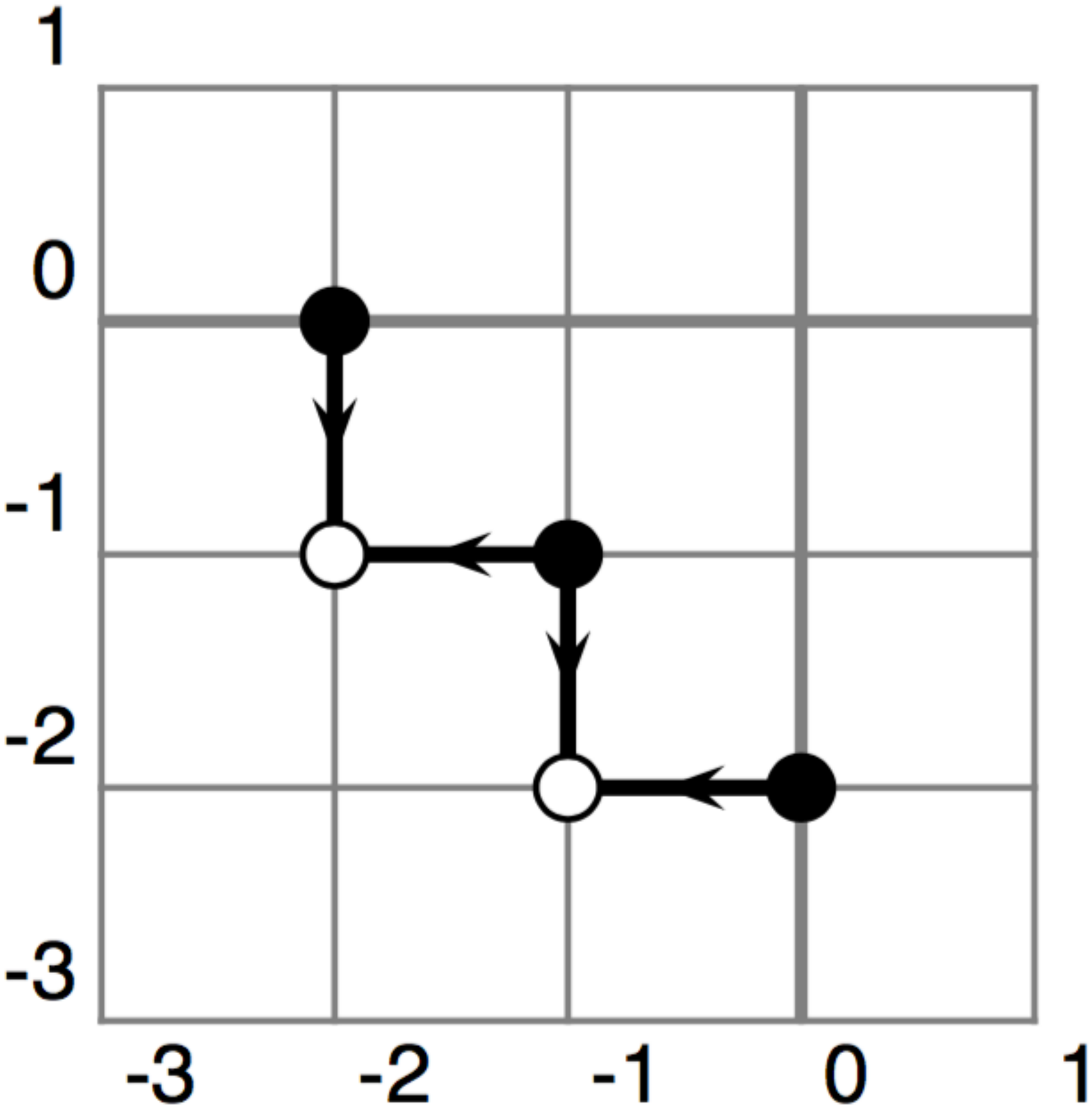}
\caption{Complexes, modulo acyclics, for $T(4,5) \cs T(2,3;4,5)$, $T(2,5)$, and $-T(2,5)$}\label{fig:ladders}
\end{figure}
\end{example}

For $\rmc_1$ the difference of the cycles $z^\pm$ form the boundary of the chain consisting of a single basis element, located at bifiltration level $(2,2)$.  For $\rmc_2$, the cycles $z^\pm$ are the boundary of the chain formed as the sum of two  basis elements, with filtration levels $(1,2)$ and $(2,1)$.  Computing $\U^2$ for $\rmc_1$ is easier, since the relevant support lines must go through the vertex $(2,2)$.  For  $\rmc_2$, the relevant support line goes through $(2,1)$ if $s\le 1$ and through $(1,2)$ is $s \ge 1$.   We leave the details to the reader.

$$\U^2_{\rmc_1, 1}(s) = -2.$$

$$
\U^2_{\rmc_2, 1}(s) =   \begin{cases}
-2 +s &\mbox{if } s \le 1\\
-s &\mbox{if } s\ge 1.
\end{cases} 
$$

Building from the  previous example, one can compute that for the knot  $$K = T(4,5) \cs -T(2,3;2,5) \cs - T(2,5)$$  the complex $\cfki$ is as illustrated in Figure~\ref{figureu0} (again, modulo acyclic summands).  This computation calls on a change of basis as well as forming the tensor product.  Given that this is the difference of two complexes with the same $\U_K$ function, the difference has $\U_K$ identically 0. 

\begin{figure}[h]
\fig{.25}{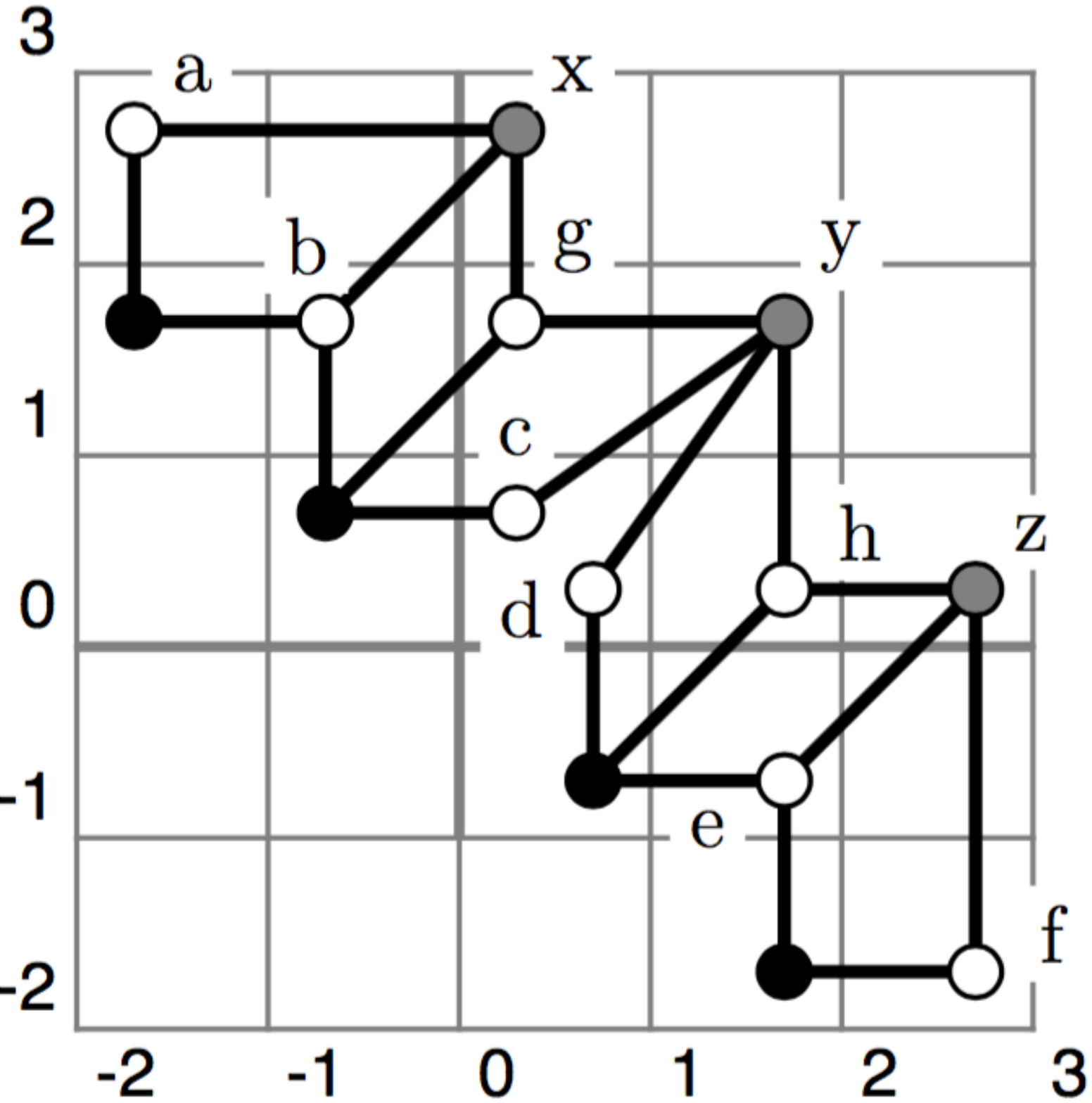}
\caption{$\cfk(T(4,5) \cs -T(2,3;2,5) \cs - T(2,5))$}
\label{figureu0}
\end{figure}

Generators of grading 0 are drawn as white circles, those of grading $-1$ are black, and those of grading 1 are gray.  The group of cycles of grading 0 is generated by four cycles:   $$\left< a+b +c, d+e+f, g+c, h+d\right>.$$  The pivoting vertices $p^-$ and $p^+$ are both $(0,0)$.  The cycles in $\calz^\pm$ are unique:  $z^- = a+b+c$ and $z^+ = d + e + f$.   The fact that these represent the same class in $H_0(\cfki)$ follows from their sum being the boundary of the chain $x + y + z$.  This is the only possible chain whose boundary offers a relation between them.  

To compute $\U^2_{K,1}(s)$, in the case that $s \le 1$ we see the support line must go through the vertex containing $z$, that is, $(2,0)$.  For this vertex, the value of $\frac{s}{2}j + (1 - \frac{s}{2})i$ is $2 - s$.  Multiplying by $-2$ (and subtracting $\U_K(1) = 0$) yields $-4 + 2s$.  

For $s \ge 0$, the essential vertex is at $(0,2)$, for which $\frac{s}{2}j + (1 - \frac{s}{2})i$ has value $s$.  Multiplying by $-2$ gives $-2s$.  Thus,

$$
\U^2_{K, 1}(s) =   \begin{cases}
-4 +2s &\mbox{if } s \le 1\\
-2s &\mbox{if } s\ge 1.
\end{cases} 
$$

Hom   constructed this knot to build a concordance class that is nontrivial  and  which cannot be detected with $\U_K$, but can be detected by $\epsilon$.  Here we see it is also detected by $\U^2_{K,t}$.  In Section~\ref{sec-genus}  we will use this example to build new ones related to bounds on the concordance genus.

\section{Subadditivity} 
In the following statement of subadditivity, the appearance of the minimum instead of the maximum is explained by the presence of the factor of $-2$ in the definition of $\U^2$.  This will become clear in the first sentence of the  proof.  We state the theorem in terms of tensor products of  $\calk$--complexes, implying the similar statement for connected sums of knots.
\begin{theorem}  
For any $\calk$--complexes $\rmc_1$ and $\rmc_2$,  $$\U^2_{\rmc_1\otimes \rmc_2, t}(t) \ge \min\{ \U^2_{\rmc_1, t}(t), \U^2_{ \rmc_2, t}(t)\} .$$
\end{theorem}

\begin{proof}  This is proved by showing that  $$(\gamma^2_{\rmc_1 \cs \rmc_2, t}(t) - \gamma_{\rmc_1 \cs \rmc_2, t}(t)) \le \max\{ (\gamma^2_{\rmc_1 , t}(t) - \gamma_{\rmc_1 , t}(t)), (\gamma^2_{  \rmc_2, t}(t) - \gamma_{  \rmc_2 , t}(t))\} .$$

From the definition of $\gamma^2$, for $r_1 = \gamma^2_{\rmc_1,t}$  there is a chain $c_1 \in \calf_{ t, \gamma_{\rmc_1}(t)} +\calf_{t,r_1}$ such that $\partial c_1 = z_1^+ - z_1^-$, with $z_1^\pm$ as described in Section~\ref{definegamma2} for the complex $\rmc_1$.  Similarly, for the complex $\rmc_2$ there is a chain $c_2$.  One has 

$$\partial( c_1 \otimes z_2^+ + z_1^- \otimes c_2) = z_1^+ \otimes z_2^+ - z_1^- \otimes z_2^-.$$
The rest is arithmetic.
\end{proof}

Note that  we consider $\U_{K,t}(t)$ in the theorem. For $\U_{K,t}(s)$ with $t\ne s$, the theorem would be false in general.

\section{Linear independence}

In Figure~\ref{figuresquare} we illustrate a $\calk$--complex which we denote $K_n$; the sides of the square are of length $2n$ and the central vertex is at the origin and has  grading 0.  The second diagram is the inverse complex, which we denote $-K_n$.  
A similar complex, in which the vertex $B$ is not in the center of the square, was used in~\cite{oss} as an example for which the nonvanishing of $\U$ implies  that the complex is not null concordant, whereas the $\epsilon$--invariant vanishes.  

For $K_n$ there are two cycles that represent generators of $H_0$:  $A +B$ and $B+C$, and these are homologous.  For $-K_n$ the classes are $A+C$ and $B$, and these are homologous.  

\begin{figure}[h] 
\fig{.2}{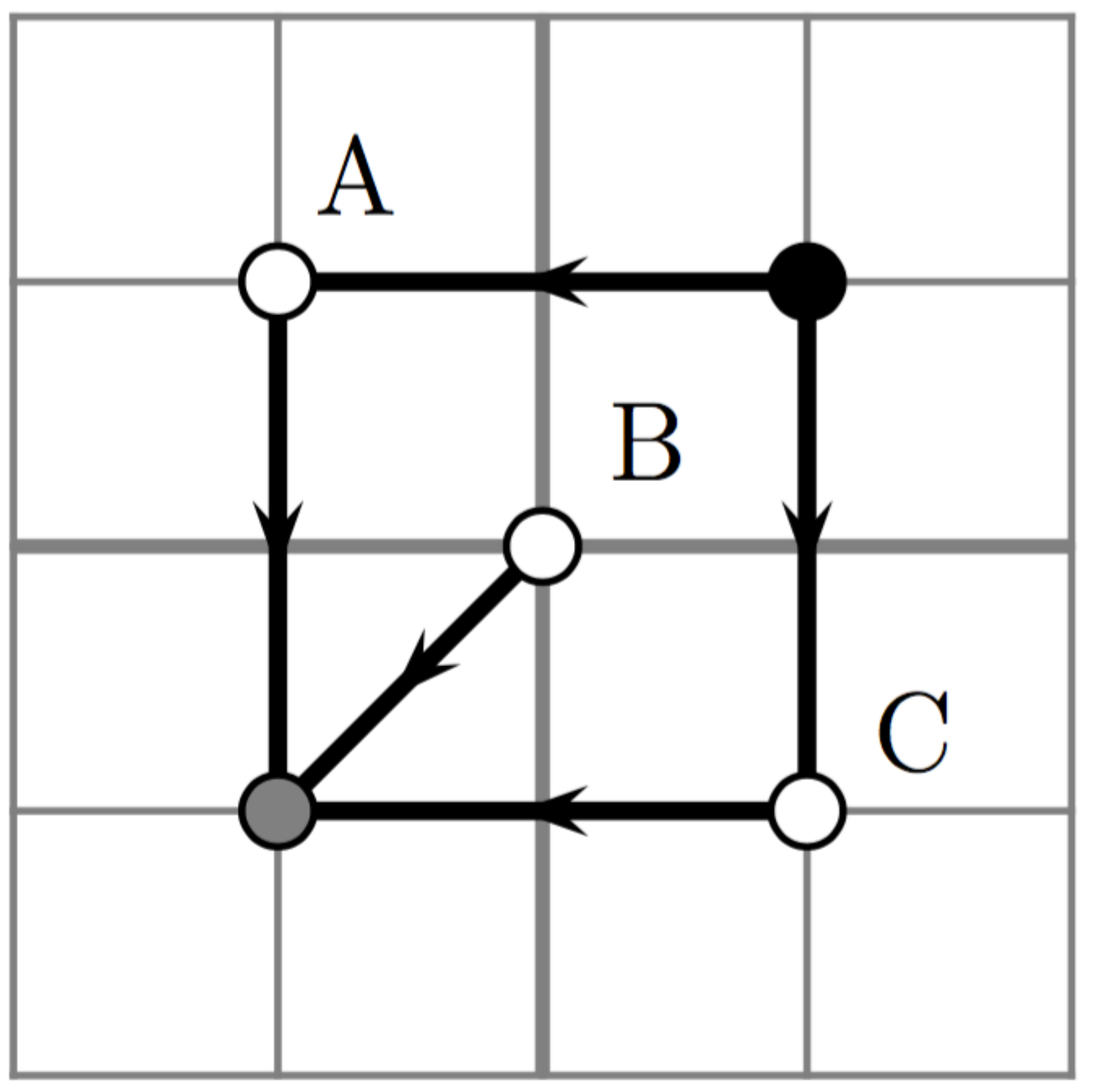} \fig{.2}{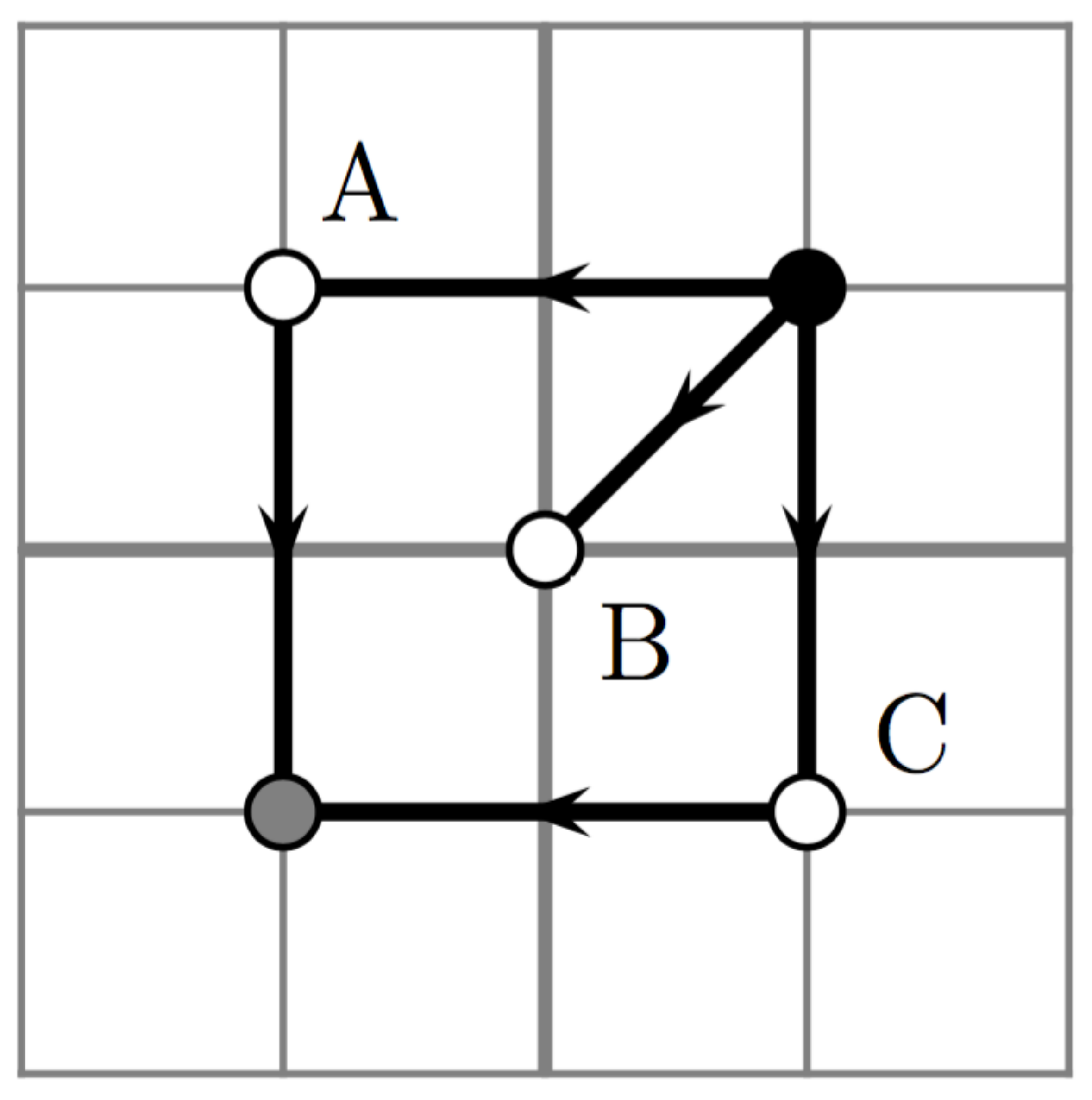}
\caption{ $\calk$--complexes $K_n$ and $-K_n$}
\label{figuresquare}
\end{figure}

For the complex $K_n$ one has $\U_{K_n}(t) = 0$ for all $t$.  
We focus on $\U^2_{K_n,1}(1)$, denoting this $\upsilon^2(K)$.  The following is an easy exercise.

\begin{theorem} $\upsilon^2(K_n) = -2n$ and $ \upsilon^2(-K_n) = 0$.
\end{theorem}

The linear independence  in concordance of these complexes follows quickly.

\begin{theorem} If $a_N >0$, then $\upsilon^2( \oplus_{n=1}^N a_nK_n) = -2N$.
\end{theorem}

\begin{proof} Write $K =   \oplus_{n=1}^N a_nK_n $.  By subadditivity,  $-2N \le \upsilon^2(K) \le 0$. If there is a strict inequality, $ -2N<   \upsilon^2(K)$, then tensoring with copies of $K_n$ or $-K_n$ for $n < N$ and applying subadditivity shows that $-2N < \upsilon^2(a_N K_N)$.  Now tensor with $-K_N$ (which has $\upsilon^2 = 0$) repeatedly (more precisely, $a_N - 1 $ times) and apply subadditivity to find that $-2N < \upsilon^2(  K_N)$, a contradiction.
\end{proof}

\section{Genus}
\label{sec-genus} Bounds on the three-genus of a knot $K$, and thus  its  concordance genus $g_c(K)$,   based on $\U_K(t)$ were presented in~\cite{livingston1, oss}.  These follow solely from the fact that for a knot of three-genus $g$, the  filtration levels of all filtered basis elements $x \in \cfki$ satisfy $|  \alg(x)- \alex(x)| \le g$.  (To be more precise, some complex representing the filtered chain homotopy class of $\cfki$ has this property.)   Thus, we have the following result corresponding to Theorem~10.1 of~\cite{livingston1}.

\begin{theorem} For every value of $t$ and for all   nonsingular points of $\U_{K,t}^{2}(s)$ as a function of $s$,
$|\Utwoprime (s)| \le g_c(K)$.
The jumps in $\Utwoprime(s)$ occur at rational numbers $\frac{p}{q}$.  For $p$ odd, $q\le g_c(K)$. If $p$ is even, $\frac{q}{2} \le g_c(K)$.
\end{theorem}

The complex $\cfk^\infty(T(3,4))$ illustrated in Figure~\ref{fig34-labeled} is called a \emph{stairway complex} of type $[1,2,2,1]$.  To set notation, its grading 0 generators will be denoted $\{a_1, a_2, a_3\} $ and its grading 1 generators $\{b_1, b_2\}$.  The $a_i$ represent a generator of homology and are all homologous.

The complex for $\cfk^\infty(-T(3,4))$ is denoted $-[1,2,2,1]$. Its grading 0 generators are now still $\{a_1, a_2, a_3\}$ but the set $\{b_1, b_2\}$ represent grading $-1$ generators.  The set of 0--cycles is one dimensional, generated by $a_1 + a_2 + a_3$. 

As illustrated in Figure~\ref{fig:ladders},  $\cfk^\infty(T(3,4) \cs - T(2,3;4,5))$ and  $-\cfk^\infty(T(2,5))$ are, modulo acyclic complexes, both given by stairway complexes, $[2,2]$ and $-[1,1,1,1]$.  

\begin{example}
We now consider the connected sum of $n$ copies of $$K =  T(3,4) \cs - T(2,3;4,5) \cs -T(2,5),$$ denoted $nK$.  We  show $g_c(nK) \ge 4n-2$. Modulo acyclic complexes, $n( T(3,4) \cs - T(2,3;4,5))$ is represented by the complex $\rmc_1 = [2,2, \ldots, 2]$ with a total of $2n$ consecutive entries $2$.  Similarly, for the knot $-nT(2,5)$, the complex is $\rmc_2 =  -[1,1,\ldots, 1]$ with $4n$ consecutive entries 1.   (For details, see~\cite[Appendix B]{hkl} or~\cite{hhn}.)  

We will denote the basis elements for  $\rmc_1$ by $A_i$ and $B_i$, and those for $\rmc_2$ by $a_i $ and $b_i$.  In the tensor product of these two complexes, there are only two types for which $\alg(\cdot) + \alex(\cdot) \le 0$. The first are of the form $b_i \otimes A_j$, with $\alg(\cdot) + \alex(\cdot) =-1$ and grading $-1$. The second are of the type $a_i \otimes A_i$ for which $\alg(\cdot) + \alex(\cdot) =0$ and grading $0$.  

Since $\partial A_i = 0$, it is easily seen that the only grading 0 cycles for which  $\alg(\cdot) + \alex(\cdot) \le 0$ are linear combinations of $(\sum_{i=1}^{2n+1}a_i) \otimes A_j$.  It then follows that $\calz^\pm$ each contain one element:  $z^- = (\sum_{i=1}^{2n+1}a_i) \otimes A_1$ and  $z^+ = (\sum_{i=1}^{2n+1}a_i) \otimes A_{n+1}$.

Suppose that $\partial w = z^- - z^+$.  Then clearly when $w$ is expressed in terms of the given basis, both $a_1 \otimes B_1$ and $a_{2n+1} \otimes B_n$ appear.  These are at filtration levels $(-2n +2, 2n)$ and $(2n, -2n+2)$, respectively.  

From these calculations, we have the following.

$$\U^2_{nK,1}(s) =
\begin{cases}
(4n-2)s -4n &\text{if\ } s \le 1\\
(-4n+2)s +4n -4&\text{if\ } s \ge 1.\\
\end{cases}
$$
As a corollary, we have
$g_c(nK) \ge 4n-2$.

Note that with care one can further show $g_c(nK)\ge 4n$ since the boundary of $a_1 \otimes B_1$ and $a_{2n+1} \otimes B_n$ must appear in the complex.  In~\cite{hom-epsilon-upsilon} Hom asked whether  another bound on the concordance genus, denoted there by $\gamma$, equals $4n$ for these knots. 

\end{example}



\begin{thebibliography}{AAA}

\bibitem{borodzik-hedden} M.~Borodzik and M.~Hedden, {\em The Upsilon function of L-space knots is a Legendre transformation},  \url{arxiv.org/abs/1505.06672}.

\bibitem{borodzik-livingston} M.~Borodzik and C.~Livingston, {\em Semigroups, d--invariants and deformations of cuspidal singular points of plane curves},   J.~Lond. Math.~Soc.~(2) {\bf 93} (2016),   439--463.

\bibitem{feller-krcatovich} P.~Feller and D.~Krcatovich, {\em On cobordisms between knots, braid index, and the Upsilon-invariant},  \url{arxiv.org/abs/1602.02637}.

\bibitem{feller-park-ray} P.~Feller, J.~Park, and A.~Ray, {\em On the Upsilon invariant and satellite knots,}   \url{arxiv.org/abs/1604.04901}.

\bibitem{hkl}  M.~Hedden, S.-G.~Kim, and C.~Livingston, {\em Topologically slice knots of smooth concordance order two}, Jour. Diff. Geom. {\bf 102} (2016), 353--393.

\bibitem{hhn} S.~Hancock, J.~Hom, and M.~Newman, {\em On the knot Floer filtration of the concordance group}, J. Knot Theory Ramifications {\bf 22} (2013) 1350084, 30 pp.

\bibitem{hendricks-manolescu} K.~Hendricks and C.~Manolescu, {\em Involutive Heegaard Floer homology}, \url{arxiv.org/abs/1507.00383}.

\bibitem{hom1} J.~Hom, {\em The knot Floer complex and the smooth concordance group},  Comment. Math. Helv. {\bf 89} (2014),  537--570. 

\bibitem{hom-survey} J.~Hom, {\em A survey on Heegaard Floer homology and concordance}, \url{arxiv.org/abs/1512.00383}.

\bibitem{hom-epsilon-upsilon} J.~Hom, {\em A note on the concordance invariants epsilon and upsilon}, Proc. Amer. Math. Soc. {\bf 144} (2016), 897--902.

\bibitem{kim-wu} M.~Kim and Z.~Wu, {\em On rational sliceness of Miyazaki's fibered, $-$amphicheiral knots,} 
\url{arxiv.org/abs/1604.04870}.

\bibitem{livingston1} C.~Livingston, {\em Notes on the knot concordance invariant Upsilon}, \url{arxiv.org/abs/1412.0254}.

\bibitem{livingston-cott} C.~Livingston and C.~Van Cott, {\em The four-genus of connected sums of torus knots,} \url{arxiv.org/abs/1508.01455}.

\bibitem{oss} P.~Ozsv{\'a}th, A.~Stipsicz, and Z.~Szab{\'o}, {\em Concordance homomorphisms from knot Floer homology}, \url{arxiv.org/abs/1407.1795}.

\bibitem{oss2} P.~Ozsv{\'a}th, A.~Stipsicz, and Z.~Szab{\'o}, {\em Unoriented knot Floer homology and the unoriented four-ball genus}, \url{arxiv.org/abs/1508.03243}.

\bibitem{wang} S.~Wang, {\em Semigroups of L-space knots and nonalgebraic iterated torus knots,}  \url{arxiv.org/abs/1603.08877}.
\end{thebibliography}
\end{document}